\documentclass[12pt]{amsart}
\usepackage{amsthm,amsmath,amssymb,enumerate,mathtools}
\newtheorem{theorem}{Theorem}[section]
\newtheorem{claim}{}[theorem]
\newtheorem{lemma}[theorem]{Lemma}
\newtheorem{problem}[theorem]{Problem}

\newtheorem{conjecture}[theorem]{Conjecture}
\theoremstyle{definition}

\newcommand{\cP}{\mathcal{P}}

\DeclareMathOperator{\cl}{cl}

\DeclareMathOperator{\GF}{GF}
\DeclareMathOperator{\AG}{AG}

\newcommand{\del}{\!\setminus\!}

\usepackage{environ, comment}
\usepackage[labelsep=period]{caption}

\title[Minimally vertically $4$-connected matroids]{Small cocircuits in minimally vertically $4$-connected matroids}

\author{James Oxley}
\address{Mathematics Department, Louisiana State University, Baton Rouge, Louisiana, USA}
\email{oxley@math.lsu.edu}

\author{Zach Walsh}
\address{Mathematics Department, Louisiana State University, Baton Rouge, Louisiana, USA}
\email{walsh@lsu.edu}

\date{\today}

\begin{document}
\maketitle


\begin{abstract} 
Halin proved that every minimally $k$-connected graph has a vertex of degree $k$.
More generally, does every minimally vertically $k$-connected matroid have a $k$-element cocircuit? 
Results of Murty and Wong give an affirmative answer when $k \le 3$.
We show that every minimally vertically $4$-connected matroid with at least six elements has a $4$-element cocircuit, or a $5$-element cocircuit that contains a triangle, with the exception of a specific non-binary $9$-element matroid.
Consequently, every minimally vertically $4$-connected binary matroid with at least six elements has a $4$-element cocircuit.
\end{abstract}


\begin{section}{Introduction}
A graph $G$ is \emph{minimally $k$-connected} if it is $k$-connected, and $G \del e$ is not $k$-connected for all edges $e$ of $G$.
While $k$-connected graphs have no vertices of degree less than $k$, Halin~\cite{Halin} proved that every minimally $k$-connected graph has a vertex of degree exactly $k$.
Since vertices of small degree are useful, for example, for facilitating inductive arguments, Halin's result was strengthened several times until finally Mader~\cite{Mader} proved a tight lower bound on the number of vertices of degree $k$ in a minimally $k$-connected graph.

Although matroids in general do not have vertices, it has been common to use cocircuits as matroid analogues of vertices.
When we seek analogues of Halin's theorem, we find that there are two widely used notions of matroid connectivity, namely (Tutte) connectivity, and vertical connectivity.

The analogue of Halin's result for $k$-connectivity was studied by Reid, Wu, and Zhou~\cite{RWZ}.
A matroid $M$ is \emph{minimally $k$-connected} if it is $k$-connected, and $M \del e$ is not $k$-connected for all elements $e$ of $M$.
Reid et al. sought to prove the following (see~\cite[Problem 14.4.9]{Oxley92}).

\begin{conjecture} \label{k-conn}
If $M$ is a minimally $k$-connected matroid with $|E(M)| \ge 2(k-1)$, then $M$ has a cocircuit of size $k$.
\end{conjecture}

Results of Murty~\cite{Murty} and Wong~\cite{Wong} prove the conjecture when $k \le 3$.
Reid et al.~\cite{RWZ} proved that the conjecture holds when $k = 4$, with a unique exception, the $9$-element rank-$4$ matroid $N_9$ that is represented over $\GF(3)$ by the matrix in Figure~\ref{matrix}.

\begin{theorem}[Reid, Wu, Zhou] \label{RWZ}
Let $M$ be a minimally $4$-connected matroid with at least six elements such that $M \ncong N_9$. 
Then $M$ has a $4$-element cocircuit.
\end{theorem}

\begin{figure}
\begin{center}
		$\begin{bmatrix}
1 & 0 & 0 & 0 & 0 & 1 & 1 & -1 & 1 \\
0 & 1 & 0 & 0 & 1 & 0 & 1 & 1  & -1 \\
0 & 0 & 1 & 0 & 1 & 1 & 0 & 1  & 1 \\
0 & 0 & 0 & 1 & -1& 1 & 1 & 0  & -1
\end{bmatrix}.$
	\end{center} 
	
	\caption{A ternary representation for $N_9$.}
	\label{matrix}
	\end{figure}

Reid et al. constructed $N_9$ from a $2$-$(9,4,3)$-design, so it is highly structured.
It has a transitive automorphism group, and each element $x$ satisfies $N_9 \del x \cong P_8$ and $N_9/x \cong \AG(2,3)\del e$.
Since $\AG(2,3)\del e$ is not binary, $N_9$ is not binary.

In the same paper, Reid et al. disprove Conjecture \ref{k-conn} for each $k \ge 5$, by finding a minimally $k$-connected matroid with $2k + 1$ elements and no $k$-cocircuits.
They conjecture that Conjecture \ref{k-conn} holds when $|E(M)| \ne 2k + 1$, so the analogue of Halin's theorem for $k$-connectivity may still hold for $k \ge 5$.


In this paper, we focus on the analogue of Halin's theorem for vertical $k$-connectivity.
This is a weaker connectivity property than $k$-connectivity, in the sense that every $k$-connected matroid is also vertically $k$-connected.
A matroid $M$ is \emph{minimally vertically $k$-connected} if it is vertically $k$-connected, and $M \del e$ is not vertically $k$-connected for all elements $e$ of $M$.
Thus Halin's theorem prompts the following natural question for vertical $k$-connectivity.

\begin{problem} \label{k}
Does every minimally vertically $k$-connected matroid with at least $2k + 2$ elements have a $k$-cocircuit?
\end{problem}

Since a graph $G$ is $k$-connected if and only if the graphic matroid $M(G)$ is vertically $k$-connected, this problem seeks a direct generalization of Halin's result.
The condition that the matroid has at least $2k + 2$ elements is necessary, due to the construction of Reid et al.~\cite{RWZ}.

When $k \le 3$, minimal vertical $k$-connectivity and minimal $k$-connectivity coincide, so the results of Murty~\cite{Murty} and Wong~\cite{Wong} affirmatively answer Problem \ref{k} when $k \le 3$.
However, for $k \ge 4$, minimal vertical $k$-connectivity is a strictly weaker condition than $k$-connectivity.
We take a step towards resolving Problem \ref{k} for $k = 4$ by showing that every minimally vertically $4$-connected matroid, except for the $9$-element matroid in Theorem~\ref{RWZ}, has a small cocircuit with special structure.

\begin{theorem} \label{minimal}
Let $M$ be a minimally vertically $4$-connected matroid with at least six elements such that $M \ncong N_9$.
Then $M$ has a $4$-cocircuit, or a $5$-cocircuit that contains a triangle.
\end{theorem}

Since no binary matroid has a $5$-cocircuit that contains a triangle, Theorem~\ref{minimal} shows that Problem~\ref{k} has an affirmative answer for binary matroids when $k = 4$.

\begin{theorem} \label{binary}
Every minimally vertically $4$-connected binary matroid with at least six elements has a $4$-cocircuit.
\end{theorem}

Despite these positive results, we conjecture that Problem \ref{k} has a negative answer when $k  = 4$ in the following strong sense.

\begin{conjecture} \label{construction}
There is an infinite family of minimally vertically $4$-connected matroids with no $4$-cocircuits.
\end{conjecture}

In a previous version of this paper we claimed to prove this conjecture, and we are grateful to the anonymous referee who pointed out an error in the proof.
While we could not fix this error, we expect that there is a clever construction that proves Conjecture \ref{construction}.


There is a key relationship between $4$-connectivity and vertical $4$-connectivity that allows us to use Theorem~\ref{RWZ} in the proof of Theorem~\ref{minimal}.
Specifically, a non-uniform matroid is minimally $4$-connected if and only if it is minimally vertically $4$-connected, and has no triangles.
Thus, by Theorem \ref{RWZ}, it suffices to prove Theorem \ref{minimal} for matroids with a triangle.
In fact, we show that every triangle of a minimally vertically $4$-connected matroid intersects a small cocircuit with special structure.

\begin{theorem} \label{minimal triangle}
Let $M$ be a minimally vertically $4$-connected matroid with a triangle $T$. Then
\begin{enumerate}[$(1)$]
\item $M$ has a $4$-cocircuit that contains exactly two elements of $T$, or

\item $M$ has a $5$-cocircuit that contains a triangle and exactly two elements of $T$, or 

\item  $|E(M)| \le 11$.
\end{enumerate}
\end{theorem}

We can even relax the condition that $M$ is minimally vertically $4$-connected, and still find a cocircuit with specific structure.

\begin{theorem} \label{main}
Let $M$ be a vertically $4$-connected matroid with a triangle $T = \{e,f,g\}$ so that none of $M\del e$, $M \del f$, $M \del g$ is vertically $4$-connected.
Then either
\begin{enumerate}[$(1)$]
\item $M$ has a cocircuit $C^*$ so that $|C^* \cap T| = 2$ and $M|C^* \cong U_{2,k} \oplus U_{2,2}$ for some $k \ge 2$, or

\item $r(M) \le 6$, and, if $M$ has no $U_{2,4}$-restrictions, then $|E(M)| \le 11$. 
\end{enumerate}
\end{theorem}

This is an analogue of a result of Tutte \cite{Tutte} called Tutte's Triangle Lemma (see \cite[Lemma 8.7.7]{Oxley}), which finds a $3$-cocircuit, or \emph{triad}, that intersects a given triangle of a $3$-connected matroid.

\begin{theorem}[Tutte] \label{TTL}
Let $M$ be a $3$-connected matroid with a triangle $\{e,f,g\}$ so that neither $M \del e$ nor $M \del f$ is $3$-connected.
Then $M$ has a triad that contains $e$ and exactly one of $f$ and $g$.
\end{theorem}

Our proof of Theorem \ref{main} follows the proof of Theorem \ref{TTL}.
Before proving Theorems \ref{minimal}, \ref{binary}, \ref{minimal triangle}, and \ref{main} in Section \ref{main proofs}, we discuss some preliminaries in Section \ref{preliminaries}.
We close by discussing several related open problems in Section \ref{open}.
\end{section}


\begin{section}{Preliminaries} \label{preliminaries}
We follow the notation of Oxley \cite{Oxley}.
Given a matroid $M$ with ground set $E$ and rank function $r$, the function $\lambda_M$ defined by 
\begin{align*}
\lambda_M(X) = r(X) + r(E - X) - r(M)
\end{align*}
is the \emph{connectivity function} of $M$.
Tutte \cite{Tutte} proved that this function is \emph{submodular}, which means that all $X, Y \subseteq E(M)$ satisfy
\begin{align*}
\lambda_M(X) + \lambda_M(Y) \ge \lambda_M(X \cup Y) + \lambda_M(X \cap Y).
\end{align*}
For a positive integer $j$, a partition $(X,Y)$ of the ground set of a matroid $M$ is a \emph{vertical j-separation} of $M$ if $\lambda_M(X) < j$ and $\min\{r(X), r(Y)\} \ge j$.
The vertical $j$-separation is \emph{exact} if $\lambda_M(X) = j - 1$.
For an integer $k$ exceeding one, a matroid $M$ is \emph{vertically k-connected} if it has no vertical $j$-separations with $j < k$.

While this is the main notion of connectivity in this paper, we also make some use of (Tutte) $k$-connectivity.
For $j \ge 1$, a partition $(X,Y)$ of the ground set of a matroid $M$ is a \emph{j-separation} of $M$ if $\lambda_M(X) < j$ and $\min\{|X|, |Y|\} \ge j$.
The $j$-separation is \emph{exact} if $\lambda_M(X) = j - 1$.
For $k \ge 2$, a matroid $M$ is \emph{k-connected} if it has no $j$-separations with $j < k$.
Every vertical $j$-separation is a $j$-separation, so every $k$-connected matroid is vertically $k$-connected.

There are two natural relationships between $k$-connectivity and vertical $k$-connectivity for non-uniform matroids.
First, a non-uniform matroid is $k$-connected if and only if it is vertically $k$-connected, and has no circuits with fewer than $k$ elements.
In particular, a non-uniform simple matroid is $4$-connected if and only if it is vertically $4$-connected and has no triangles.
Second, a non-uniform matroid $M$ is $k$-connected if and only if $M$ and $M^*$ are both vertically $k$-connected.
These two relationships combine to show that a non-uniform vertically $k$-connected matroid has no cocircuits with fewer than $k$ elements.
In particular, a non-uniform vertically $4$-connected matroid has no triads.
We need one more useful property of vertically $4$-connected matroids.

\begin{lemma} \label{4-pt line}
Let $M$ be a vertically $4$-connected matroid. 
For a subset $X$ of $E(M)$ and an element $x$ of $X$, if $M|X$ is isomorphic to $U_{2,4}$ or $U_{2,4} \oplus_2 U_{2,4}$, then $M \del x$ is vertically $4$-connected.
\end{lemma}
\begin{proof}
Suppose that $M|X \cong U_{2,4}$.
Let $(A,B)$ be a vertical $3$-separation of $M \del x$.
Then $A$ or $B$ contains at least two elements of $X - x$ and thus spans $x$, so either $(A \cup x, B)$ or $(A, B \cup x)$ is a vertical $3$-separation of $M$, a contradiction.

Suppose that $M|X \cong U_{2,4} \oplus_2 U_{2,4}$.
Then $X$ is the union of disjoint triangles $T_1$ and $T_2$.
Let $x \in T_1$, and let $(A,B)$ be a vertical $3$-separation of $M \del x$.
Then, without loss of generality, $A$ contains at least three elements of $(T_1 \cup T_2) - x$.
As $A$ does not span $x$, we deduce that $A \cap X = T_2$, so $B$ spans $x$, a contradiction.
\end{proof}

\end{section}


\begin{section}{The proofs of the main results} \label{main proofs}
In this section, we prove Theorems \ref{minimal}, \ref{binary}, \ref{minimal triangle}, and \ref{main}.
We first prove a lemma about the interaction of vertical $3$-separations of $M \del a$ and $M \del b$, where $a$ and $b$ are in a common triangle.

\begin{lemma} \label{technical}
Let $M$ be a vertically $4$-connected simple matroid with a triangle $T = \{a, b, c\}$.
Let $(X_a, Y_a)$ and $(X_b, Y_b)$ be exact vertical $3$-separations of $M \del a$ and $M \del b$, respectively, so that $b \in X_a$ and $a \in X_b$. 
Then either
\begin{enumerate}[$(i)$]
\item $M$ has a cocircuit $C^*$ so that $|C^* \cap T| = 2$ and $M|C^* \cong U_{2,k} \oplus U_{2,2}$ for some $k \ge 2$, or

\item $r(X_a \cap Y_b) = 2$ and $r(X_b \cap Y_a) = 2$, and either $|X_a \cap X_b| = 1$, or $r(Y_a \cap Y_b) \le 2$ and $X_a \cap X_b \ne \varnothing$.
\end{enumerate}
\end{lemma}
\begin{proof}
Suppose that neither $(i)$ nor $(ii)$ holds.
Because $M$ has a triangle but has rank at least three, $M$ is not uniform.
Thus, it has no cocircuits of size less than four.
Observe that $c \in Y_a \cap Y_b$ otherwise $(X_a \cup a, Y_a)$ or $(X_b \cup b, Y_b)$ is a vertical $3$-separation of $M$.

\begin{claim} \label{3-conn}
$M \del a,b$ is $3$-connected.
\end{claim}
For some $j \in\{1,2\}$, let $(A, B)$ be a $j$-separation of $M \del a,b$ with $c \in A$ where $A$ is closed.
Then $\lambda_M(A \cup \{a, b\}) \le j$, and so either $r(A \cup \{a, b\}) \le j$ or $r(B) \le j$.
As $M$ is simple and has no triads, $r(A \cup \{a,b\}) \ge 2$ and $r(B) \ge 2$.
Thus $j = 2$, and $r(A \cup \{a,b\}) = 2$ or $r(B) = 2$.
Suppose that $r(B) = 2$.
Then $(i)$ holds with $C^* = B \cup \{a, b\}$.
If $r(A \cup \{a, b\}) = 2$, then $M|(A' \cup \{a,b\}) \cong U_{2,4}$ for a $2$-element subset $A'$ of $A$, so, by Lemma \ref{4-pt line}, $M \del a$ is vertically $4$-connected, a contradiction.
Thus \ref{3-conn} holds.

We make repeated use of the following.

\begin{claim} \label{closure}
$a \in \cl(X_b - a)$ and $b \in \cl(X_a - b)$.
\end{claim}
If $a \notin \cl(X_b - a)$, then $\lambda_{M\del b}(X_b - a) \le 2$.
But then the complement of $X_b - a$ in $M \del b$ contains $a$ and $c$ and thus spans $b$, so $\lambda_M(X_b - a) \le 2$.
This implies that $r_M(X_b - a) = 2$, so $(i)$ holds with $C^* = (X_b \cup b) - \cl(Y_b)$.
Thus \ref{closure} holds.

The following is due to the submodularity of $\lambda_{M \del a,b}$.

\begin{claim} \label{uncrossing}
\begin{enumerate}[$(i)$]
\item $\lambda_{M \del a,b}(Y_a \cap Y_b) + \lambda_{M \del a,b}(X_a \cap X_b) \le 4$, and

\item $\lambda_{M \del a,b}(X_a \cap Y_b) + \lambda_{M \del a,b}(X_b \cap Y_a) \le 4$.
\end{enumerate}
\end{claim}
We prove \ref{uncrossing}$(ii)$.
The proof of \ref{uncrossing}$(i)$ is similar.
We have $\lambda_{M \del a,b}(X_a - b) \le \lambda_{M \del a}(X_a) = 2$. 
Similarly, $\lambda_{M \del a,b}(Y_b) \le 2$.
Then 
\begin{align}
4 &\ge \lambda_{M \del a,b}(X_a - b) + \lambda_{M \del a,b}(Y_b) \\
&\ge \lambda_{M \del a,b}(X_a \cap Y_b) + \lambda_{M \del a,b}((X_a \cup Y_b) - b) \\
&= \lambda_{M \del a,b}(X_a \cap Y_b) + \lambda_{M \del a,b}(X_b \cap Y_a),
\end{align}
where (2) holds by the submodularity of $\lambda_{M \del a,b}$, and (3) holds because the complement of $(X_a \cup Y_b) - b$ in $M \del a,b$ is $X_b \cap Y_a$.
Thus \ref{uncrossing}$(ii)$ holds.

We now determine the ranks of $X_a \cap Y_b$ and $X_b \cap Y_a$.

\begin{claim} \label{ranks}
$r(X_a \cap Y_b) = r(X_b \cap Y_a) = 2$.
\end{claim}
If $X_a \cap Y_b = \varnothing$, then $X_a - b \subseteq X_b$. 
But then $X_b$ spans $b$, so $(X_b \cup b, Y_b)$ is a vertical $3$-separation of $M$, a contradiction.
Suppose that $X_a \cap Y_b = \{d\}$.
Then $|X_a \cap X_b| \ge 2$, or else $(i)$ holds with $C^* = X_a \cup a$.
Similarly, $|Y_a \cap Y_b| \ge 2$, or else $\{d ,c, b\}$ is a triad of $M$.
Then \ref{3-conn} implies that $\lambda_{M \del a,b}(X_a \cap X_b) \ge 2$ and $\lambda_{M \del a,b}(Y_a \cap Y_b) \ge 2$.
By \ref{uncrossing}$(i)$, it follows that $\lambda_{M \del a,b}(Y_a \cap Y_b) = 2$.
The complement of $Y_a \cap Y_b$ in $M \del a,b$ contains $X_a - b$ and $X_b - a$, and thus spans $a$ and $b$, by \ref{closure}.
This implies that $\lambda_M(Y_a \cap Y_b) = 2$, and so $r(Y_a \cap Y_b) = 2$, since $M$ is vertically $4$-connected.
But then $(i)$ holds with $C^* = (Y_b \cup b) - \cl(X_b)$, a contradiction.
Thus, $|X_a \cap Y_b| \ge 2$, and by symmetry, $|X_b \cap Y_a| \ge 2$.

Now $\lambda_{M \del a,b}(X_a \cap Y_b) \ge 2$ and $\lambda_{M \del a,b}(X_b \cap Y_a) \ge 2$, by \ref{3-conn}.
Then \ref{uncrossing}$(ii)$ implies that $\lambda_{M \del a,b}(X_a \cap Y_b) = 2$.
The complement of $X_a \cap Y_b$ in $M \del a,b$ contains $X_b - a$, and thus spans $a$, by \ref{closure}.
But then it also spans $b$ since $\{a, b, c\}$ is a triangle, and so $\lambda_{M}(X_a \cap Y_b) = 2$.
Since $M$ is vertically $4$-connected and $|X_a \cap Y_b| \ge 2$, it follows that $r(X_a \cap Y_b) = 2$.
By symmetry, $r(X_b \cap Y_a) = 2$ as well.
Thus \ref{ranks} holds.

Suppose that $X_a \cap X_b = \varnothing$.
Then since $r(X_a \cap Y_b) = 2$, outcome $(i)$ holds with $C^* = (X_a \cup a) - \cl(Y_a)$, a contradiction.
Thus $|X_a \cap X_b| \ge 1$.
Since $(ii)$ does not hold, we have $r(X_a \cap X_b) \ge 2$ and $r(Y_a \cap Y_b) \ge 3$.
Then \ref{3-conn} implies that $\lambda_{M \del a,b}(X_a \cap X_b) \ge 2$ and $\lambda_{M \del a,b}(Y_a \cap Y_b) \ge 2$.
By \ref{uncrossing}$(i)$, it follows that $\lambda_{M \del a,b}(Y_a \cap Y_b) = 2$.
The complement of $Y_a \cap Y_b$ in $M \del a,b$ contains $X_b - a$ and $X_a - b$, and thus spans $a$ and $b$, by \ref{closure}.
Thus, $\lambda_M(Y_a \cap Y_b) = 2$, and so $r(Y_a \cap Y_b) = 2$, a contradiction.
\end{proof}


The following easily implies Theorem \ref{main}. 
We add an extra condition to outcome (2) to help deal with matroids on at most $11$ elements.

\begin{theorem} \label{main extra}
Let $M$ be a vertically $4$-connected matroid with a triangle $T = \{e,f,g\}$ so that none of $M\del e$, $M \del f$, $M \del g$ is vertically $4$-connected.
Then either
\begin{enumerate}[$(1)$]
\item $M$ has a cocircuit $C^*$ so that $|C^* \cap T| = 2$ and $M|C^* \cong U_{2,k} \oplus U_{2,2}$ for some $k \ge 2$; or

\item $r(M) \le 6$, and if $M$ has no $U_{2,4}$-restrictions, then $|E(M)| \le 11$, while if $|E(M)| = 11$, then $M$ has disjoint triangles $T_1$ and $T_2$, neither of which is $T$.
\end{enumerate}
\end{theorem}
\begin{proof}
Suppose that (1) does not hold for $M$.
If $M$ is uniform, then $r(M) = 2$ and outcome (2) holds, so we may assume that $M$ is non-uniform.

\begin{claim} \label{small sep}
If $M \del e$ has an exact vertical $j$-separation $(X,Y)$ with $j \in \{1,2\}$, then $(2)$ holds.
\end{claim}
It is easy to show that if $j = 1$, then $M \cong U_{2,3}$.
Suppose $j = 2$.
Then $r(X) \le 2$, otherwise $(X, Y \cup e)$ is a vertical $3$-separation of $M$, a contradiction.
Similarly, $r(Y) \le 2$.
Since $\lambda_{M \del e}(X) = 1$, this implies that $r(M) = 3$.
Also, if $M$ has no $U_{2,4}$-restrictions, then $|E(M)| \le 7$.
Thus \ref{small sep} holds.

Let $(A_e, B_e)$, $(A_f, B_f)$, and $(A_g, B_g)$ be vertical $3$-separations of $M \del e$, $M\del f$, and $M \del g$, respectively, so that $f \in A_e \cap A_g$ and $e \in A_f$.
Then $g \in B_e \cap B_f$ and $e \in B_g$.
By \ref{small sep}, we may assume that each of the designated vertical $3$-separations is exact, or else (2) holds.
We apply Lemma \ref{technical} with
\begin{itemize}
\item $(X_a, Y_a) = (A_e, B_e)$ and $(X_b, Y_b) = (A_f, B_f)$,

\item $(X_a, Y_a) = (B_e, A_e)$ and $(X_b, Y_b) = (B_g, A_g)$, and

\item $(X_a, Y_a) = (B_f, A_f)$ and $(X_b, Y_b) = (A_g, B_g)$.
\end{itemize}
In each case, outcome $(ii)$ of Lemma \ref{technical} holds.
First, suppose that at least two of $|A_e \cap A_f|$, $|B_e \cap B_g|$, and $|B_f \cap A_g|$ are equal to one.
Up to relabeling $e$, $f$, and $g$, we may assume that $|A_e \cap A_f| = 1$ and $|B_e \cap B_g| = 1$.
Note that $r(A_e \cap B_f) = 2$ and $r(B_e \cap A_g) = 2$.
Then $A_e - f$ and $B_e - g$ are each a union of a rank-$2$ set and a rank-$1$ set, so $r(A_e) \le 4$ and $r(B_e) \le 4$.
Since $\lambda_{M \del e}(A_e) = 2$, this implies that $r(M) \le 6$.
Also, if $M$ has no $U_{2,4}$-restrictions, then $|A_e| \le 5$ and $|B_e| \le 5$, so $|E(M)| \le 11$.
If $|E(M)| = 11$, then $A_e$ and $B_e$ each contain a triangle, so (2) holds.

Second, suppose that fewer than two of $|A_e \cap A_f|$, $|B_e \cap B_g|$, and $|B_f \cap B_g|$ are equal to one.
Then, by Lemma \ref{technical}, at least two of $r(B_e \cap B_f)$, $r(A_e \cap A_g)$, $r(A_f \cap B_g)$ are at most two.
Up to relabeling $e$, $f$, and $g$, we may assume that $r(B_e \cap B_f) \le 2$ and $r(A_e \cap A_g) \le 2$.
But then $A_e$ and $B_e$ are each the union of two sets of rank at most two, so $r(A_e) \le 4$ and $r(B_e) \le 4$. 
Since $\lambda_{M \del e}(A_e) = 2$, this implies that $r(M) \le 6$.

Now assume that $M$ has no $U_{2,4}$-restrictions.
Then $|A_e - f| \le 5$ and $|B_e - g| \le 5$, and so $|E(M)| \le |A_e - f| + |B_e - g| + |T| \le 13$.
We show that $|E(M)| \le 11$.
Suppose that $|A_e| = 6$.
Since $r(A_e \cap B_g) = 2$ and $r(A_e \cap A_g) \le 2$, each of these sets is a triangle.
So, $A_e$ is the disjoint union of triangles $T_1$ and $T_2$, where $f \in T_2$.
Each of $A_f$ and $B_f$ contains an element of $T_2$, or else $A_f$ or $B_f$ spans $f$.
As $|A_e \cap B_f| \ge 2$, it follows that $|T_1 \cap B_f| \ge 1$.
If $|T_1 \cap B_f| \ge 2$, then $M$ has a $U_{2,4}$-restriction consisting of $T_1$ and the element in $T_2 \cap B_f$, since $r(A_e \cap B_f) = 2$.
Thus, $|T_1 \cap B_f| = 1$.
Let $s \in T_1 \cap B_f$.
Then $A_f$ spans $s$, and so (1) holds with $C^* = (B_f \cup f) - \cl(A_f)$, a contradiction.
Thus, $|A_e| \le 5$, and a similar argument shows that $|B_e| \le 5$, so $|E(M)| \le 11$.
It is clear from the argument that if $|A_e| = |B_e| = 5$, then each contains a triangle, and so (2) holds.
\end{proof}


\begin{proof}[Proof of Theorem~\ref{minimal triangle}]
As $M$ is minimally vertically $4$-connected, Lemma \ref{4-pt line} implies that $M$ has no $U_{2,4}$-restriction.
Thus, in Theorem~\ref{main extra}, we see that $k \le 3$ in outcome $(1)$, while $|E(M)| \le 11$ in outcome $(2)$.
The theorem follows.
\end{proof}

Theorem \ref{minimal} directly implies Theorem~\ref{binary}, since a circuit and a cocircuit in a binary matroid meet in an even number of elements.
Thus we need only prove Theorem~\ref{minimal}.
To do this, we must investigate matroids with at most $11$ elements.

\begin{proof}[Proof of Theorem~\ref{minimal}]
Let $M$ be a minimally vertically $4$-connected matroid with at least six elements.
Suppose that $M$ has no $4$-cocircuits and no $5$-cocircuits containing a triangle.
By Theorem~\ref{minimal triangle}, we must have that $|E(M)| \le 11$.
Theorem~\ref{RWZ} implies that $M$ has a triangle $T$.

\begin{claim} \label{exact}
$M$ has no element $a$ for which $M \del a$ has an exact vertical $j$-separation for $j \in \{1,2\}$.
\end{claim}
Let $(X,Y)$ be an exact vertical $j$-separation of $M \del a$ with $j \in \{1,2\}$.
It is easy to show that if $j = 1$, then $M \cong U_{2,3}$, a contradiction.
Thus $j = 2$.
If $r(X) \ge 3$, then $(X, Y \cup a)$ is a vertical $2$- or $3$-separation of $M$, a contradiction.
So $r(X) = 2$, and, similarly, $r(Y) = 2$.
By Lemma~\ref{4-pt line}, this implies that $|E(M)| \le 7$.
Since $\lambda_{M \del a}(X) = 1$, it also implies that $r(M) = 3$.
However, by Lemma~\ref{4-pt line}, $M$ has no disjoint triangles in a common plane, so $|E(M)| = 6$.
Since $M$ is a $6$-element rank-$3$ matroid that is $3$-connected, it is isomorphic to $M(K_4)$, $\mathcal W^3$, $Q_6$, $P_6$, or $U_{3,6}$ \cite[Corollary 12.2.19]{Oxley}, so $M$ has a $4$-cocircuit, a contradiction.
Thus \ref{exact} holds.

A consequence of \ref{exact} is that $M$ has no cocircuits with fewer than four elements.
For all $a$ in $E(M)$, the deletion $M \del a$ has an exact $3$-separation $(A_a, B_a)$.
Since each of $A_a$ and $B_a$ must contain a cocircuit of $M \del a$ but $M$ has no cocircuits of size less than five, $|E(M)| \ge 9$.
As $M \del a$ has no exact $2$-separations, we may assume that $T \subseteq B_a$.
Then $|E(M)| \ne 9$, otherwise $|A_a| = 4 = |B_a|$ and $B_a \cup a$ is a $5$-cocircuit containing a triangle, a contradiction.


Next suppose that $|E(M)| = 10$.
Then $|B_a| = 5$ and $|A_a| = 4$.
Also, $A_a$ does not contain a triangle, or else $A_a \cup a$ is a $5$-cocircuit of $M$ that contains a triangle.
Then $M|A_a \cong U_{3,4}$ otherwise $M|A_a$ has a coloop $x$ and $(A_a - x, B_a \cup x)$ is a vertical $3$-separation of $M \del a$, so $(A_a - x) \cup a$ is $4$-cocircuit of $M$, a contradiction.
Similarly, each element $x \in B_a - T$ is in $\cl(B_a - x)$, or else $(B_a - x) \cup a$ is a $5$-cocircuit of $M$ that contains $T$.
This implies that $r(B_a) = 3$.
As $r(A_a) = 3$, it follows that $r(M) = 4$.
So, for each $a \in E(M) - T$, there is a pair $P_a$ of elements so that $r(T \cup P_a) = 3$.
Since $M|A_a \cong U_{3,4}$ and $a \notin \cl(B_a)$, it follows that $P_a \ne P_b$ if $a \ne b$ otherwise $A_a - b$ spans $A_a \cup a$, a contradiction.
Let $\cP$ be the set of pairs $P$ for which $r(P \cup T) = 3$.
Then $|\cP| \ge |E(M) - T| = 7$.
However, the sets in $\cP$ are pairwise disjoint, or else $M$ has a $6$-element plane and thus has a $4$-cocircuit.
Since $|\cP| \ge 7$, this implies that $|E(M)| \ge 17$, a contradiction.


Finally, suppose that $|E(M)| = 11$.
By Theorem \ref{main extra}, $M$ has disjoint triangles $T_1$ and $T_2$.
Let $X = E(M) - (T_1 \cup T_2)$.
For each $a \in X$, let $(A_a, B_a)$ be a vertical $3$-separation of $M \del a$ so that $T_1 \subseteq A_a$,  while $T_2$ is contained in $A_a$ or $B_a$.
If $T_2 \subseteq A_a$ for some $a \in X$, then, by Lemma \ref{4-pt line}, $r(A_a) \ge 4$, so $r(M) \ge 5$.
If $T_2 \subseteq B_a$ for some $a \in X$, then $|A_a| = |B_a| = 5$.
Then $M|A_a$ and $M|B_a$ each have no coloops, or else $M$ has a $5$-cocircuit containing a triangle.
This implies that $r(M) = 4$.
Thus, either $T_2 \subseteq A_a$ for all $a \in X$, or $T_2 \subseteq B_a$ for all $a \in X$.
If the former holds, then $M|B_a \cong U_{3,4}$ for each $a \in X$.
It follows that $r(X) = 3$, so $r(B_a \cup a) = r(B_a)$, a contradiction.
We deduce that $T_2 \subseteq B_a$ for all $a \in X$.
This implies that, for each $a \in X$, there is a pair $P_a$ for which $r(T_2 \cup P_a) = 3$.
If $P_a = P_b$ for distinct $a,b \in X$, then $T_2 \cup P_a$ and its complement in $M$ each have rank $3$, so $M$ is not vertically $4$-connected, a contradiction.
Let $\cP$ be the set of pairs $P$ for which $r(T_2 \cup P) = 3$.
Then $|\cP| \ge |X| = 5$.
However, the sets in $\cP$ are pairwise disjoint, or else $T_2$ is contained in a $6$-element plane that is disjoint from $T_1$, so $T_1$ is contained in a $5$-cocircuit of $M$.
Since $|\cP| \ge 5$, this implies that $|E(M)| \ge 16$, a contradiction.
\end{proof}

\end{section}


\begin{section}{Open Problems} \label{open}
Our results lead to some natural open problems in several different directions.
First, we conjecture that Theorem~\ref{minimal} can be strengthened so that the cocircuit contains a specific element of $T$, as in Theorem~\ref{TTL}.

\begin{conjecture} 
Let $M$ be a minimally vertically $4$-connected matroid with a triangle $T = \{e,f,g\}$. Then
\begin{enumerate}[$(1)$]
\item $M$ has a $4$-cocircuit that contains $e$ and exactly one of $f$ and $g$, or

\item $M$ has a $5$-cocircuit that contains a triangle and $e$, and exactly one of $f$ and $g$, or 

\item  $|E(M)| \le 11$.
\end{enumerate}
\end{conjecture}

This would be helpful for trying to find a tight lower bound on the number of $4$-cocircuits and $5$-cocircuits containing a triangle in a minimally vertically $4$-connected matroid.

Second, we conjecture that Theorem \ref{binary} extends to vertical $k$-connectivity with $k \ge 5$ due to restrictions on the interaction between circuits and cocircuits in binary matroids.

\begin{conjecture}
For each integer $k \ge 5$, every minimally vertically $k$-connected binary matroid with at least $2k + 2$ elements has a $k$-cocircuit.
\end{conjecture}

We close with the following extension of Conjecture \ref{construction}, which would show that Problem \ref{k} has a negative answer when $k \ge 4$.

\begin{conjecture}
For each integer $k \ge 4$, there is an infinite family of minimally vertically $k$-connected matroids with no $k$-cocircuits.
\end{conjecture}

\end{section}

\begin{section}*{Acknowledgements}
We thank an anonymous referee for pointing out an error in our construction of an infinite family of minimally vertically $4$-connected matroids with no $4$-cocircuits in a previous version of this paper.
\end{section}

\end{document}